\documentclass[a4paper,12pt]{amsart}
\usepackage{amssymb}
\usepackage[latin1]{inputenc}
\usepackage{amsmath}
\usepackage{graphics}
\usepackage[dvips]{graphicx}
\newtheorem{thm}{Theorem}
\newtheorem{ej}{Example}

\newtheorem{prop}{Proposition}
\newtheorem{defn}{Definition}
\newtheorem{obs}{Observation}

\def\C{{\mathbb C}}

\def\D{{\mathbb D}}

\makeatletter

\begin{document}

\title {Preschwarzian derivative for logharmonic mappings}

\author{V. Bravo\and R. Hern\'andez \and O. Venegas }
\thanks{The
authors were partially supported by Fondecyt Grants \# 1190756.
\endgraf  {\sl Key words:} pre-Schwarzian derivative, logharmonic mappings, univalence criterion.
\endgraf {\sl 2010 AMS Subject Classification}. Primary: 30C55, 30G30;\,
Secondary: 31C05.}
\address{Departamento de Ciencias Matem\'{a}ticas y F\'{\i}sicas. Facultad de Ingenier\'{\i}a\\ Universidad Cat\'olica de
Temuco.} \email{ovenegas@uct.cl}
\address{Facultad de Ingenier\'ia y Ciencias\\
Universidad Adolfo Ib\'a\~nez\\
Av. Padre Hurtado 750, Vi\~na del Mar, Chile.}
\email{rodrigo.hernandez@uai.cl}
\address{Facultad de Ingenier\'ia y Ciencias\\
Universidad Adolfo Ib\'a\~nez\\
Av. Padre Hurtado 750, Vi\~na del Mar, Chile.}
\email{victor.bravo.g@uai.cl}
\address{}\email{}
\address{}\email{}

\begin{abstract}

We introduce a new definition of pre-Schwarzian derivative for logharmonic mappings and basic properties such as the chain rule, multiplicative invariance and affine invariance are proved for these operators. It is shown that the pre-Schwarzain is stable only with respect to rotations of the identity. A characterization is given for the case when the pre-Schwarzian derivative is holomorphic.
\end{abstract}

\maketitle

\section{Introduction}

Let $f$ be a locally univalent analytic function defined on a simply
connected domain $\Omega\subset \C$. The \emph{pre-Schwarzian
derivative} of $f$ is defined by
\begin{equation}\label{eq-def-preSch-analytic}
Pf=\frac{f''}{f'}\,.
\end{equation}

\par
It is well known that $P(A\circ f)=Pf$ for all linear (or affine)
transformations $A(w)=aw+b$, $a\neq 0$.
\par\smallskip
These are just particular cases of the \emph{chain rule} for the pre-Schwarzian derivative:
\begin{equation*}
P(g\circ f)= (Pg\circ f)f^\prime+Pf,
\end{equation*}
which hold whenever the composition $g\circ f$ is defined.
\par
The pre-Schwarzian derivative of locally
univalent analytic mappings $f$ is a widely used tool in the study
of geometric properties of such functions. For instance, it can be
used to get either necessary or sufficient conditions for the global
univalence, or to obtain certain geometric conditions on the range
of $f$. More specifically, it is well known that any
\emph{univalent} analytic transformation $f$ in the unit disk $\D$
satisfies the sharp inequality
\[
|Pf(z)|\leq \frac{6}{1-|z|^2}\,, \quad z\in\D\,.
\]
Also, one of the most used univalence criterion for locally univalent holomorphic mappings was showed by J. Becker in  \cite{B72}, which asserts that
if $f$ is locally univalent (and analytic) in $\D$ and
\[
\sup_{z\in\D} |zPf(z)|\, (1-|z|^2)\leq 1\,,
\]
then $f$ is univalent in the unit disk. Becker and Pommerenke \cite{BP} proved later that the constant $1$ is sharp.
\par\smallskip

Moreover, L. Ahlfors in \cite{Ah} gave a generalization of these results, this criterion can be formulated as: Let $f : \D \rightarrow \C$ be a holomorphic function with $f'\neq 0$. Let $c \in \C$ with $|c| < 1$, $c\neq -1$, and assume that
 $$\left|z\dfrac{f''}{f'}(z)(1-|z|^2)+c|z|^2\right|\leq1,\quad\quad \forall z\in\D,$$  then $f$ is univalent in $\D$.

For more properties and results related to the pre-Schwarzian derivative of locally univalent analytic mappings, we
refer the reader to the monographs \cite{Dur-Univ}, \cite{GK}, or
\cite{P}.
\par\medskip

More generally, the definition of a pre-Schwarzian derivative for harmonic mapping on the complex plane has been extended, obtaining results such as the Becker's criterion, among others. We recommend the reader to review the paper \cite{RH-MJ}. Recently I. Efraimidis et al. in \cite{EFHV}, they provided a definition of this operator in the context of several complex variables and for functions called pluriharmonic. However, for functions called \textit{logharmonic} we have not found a reference for the pre-Schwarzian derivative, even though the Schwarzian derivative has been defined for this type of functions (see \cite{MPW13}), making an extension of the definition given by Chuaqui et al. in \cite{CHDO} for harmonic mappings.


\par

The main object of this work is to give a coherent definition for the pre-Schwarzian derivative, $ P_f $ defined by equation (\ref{eq-def-preSch-analytic}), of logharmonic functions $ f $ defined in a simply connected domain. As in the case of harmonic mapping, we will justify this definition in two different ways, but which are closely related. The first consists in considering that the pre-Schwarzian derivative is the derivative of the Jacobian Logarithm of $ f $, while the second refers to the approximation of $ f $ by means of \textit{affine logharmonic} mapping in the same way that Tamanoi does for the Schwarzian derivative, see \cite{T}.
\par


We prove that not only the classical properties of the pre-Schwarzian derivative (chain rule, characterization of affine logharmonic mappings, etc.) are held for this type of functions, but also that we determine that $ P_ {A \circ f} = P_f $ for certain affine logarmonic mappings and we characterize the solutions of the equation $ P_f = P_F $.
\par

In section 2 we shall give a definition of the pre-Schwarzian derivative, $P_f$, for locally univalent logharmonic mappings $f=h\overline{g}$ defined in any simple connected domain. Moreover, we proved the chain rule, characterize the logharmonic mappings with $P_f=0$, and we show that $P_f(z_0)=P(h^ag^b)(z_0)$ for some especial choose of the complex numbers $a$ and $b$, which depends of $z_0$ (this is the same case in the harmonic scenario, see \cite{RH-MJ}). In section 3 we gave the complementary derivation of $P_f$ using the best \textit{affine logharmonic} mappings approximation for a given logharmonic mapping $f$ in the same way to the Tamanoi's work in \cite{T}.
\par

Finally, we are able to generalize the univalence criteria of Becker and Ahlfors for logharmonic mappings, in the same way of the generalization of this criteria for the case when $f$ is sense-preserving harmonic mappings. The reader can find these references in \cite{VBRHOV1} and \cite{RH-MJ}.

\par\medskip
\subsection{Logharmonic mappings} A logharmonic mapping defined in the unit disc is a solution of the nonlinear elliptic partial differential equation \begin{equation}\label{def-log}\overline{f_{\overline{z}}(z)}=\omega(z)\left(\dfrac{\overline{f(z)}}{f(z)}\right)f_z(z),\end{equation} where $\omega$ is the second complex dilatation of $f$ and $|\omega(z)|<1$ for all $z\in \D$. Thus, the Jacobian $J_f$ of $f$ is given by
\begin{equation}\label{jacobiano}J_f=|f_z|^2-|f_{\overline{z}}|^2=|f_z|^2(1-|\omega|^2),
\end{equation} which is positive, and therefore, every non-constant logharmonic mapping is sense-preserving and open in $\D.$ If $f$ is a non-constant logharmonic mapping defined in $\D$ and vanishes only at the origin, then $f$ has the representation \begin{equation*}\label{representation-Log} f(z)=z^m|z|^{2\beta m}h(z)\overline{g(z)},\end{equation*} where $m$ is a nonnegative integer, Re$\{\beta\}>-1/2$, and $h$ and $g$ are analytic mappings in the unit disc, such that $g(0)=1$ and $h(0)\neq0$ (see \cite{AB88}). In particular, this class of logharmonic mappings with $m=1$ has been widely studied in the last years. For more details the reader can read the summary paper in \cite{AR}. If $f$ is a nonvanishing mapping in $\D$, then $f$ can be expressed as
\begin{equation*} f(z) = h(z)\overline{g(z)},\end{equation*}
where $h(z)$ and $g(z)$ are non-vanishing analytic functions in $\D.$ But, $f$ is locally univalent mapping, then the second complex dilatation $\omega$ is a Schwarz's function given by \begin{equation}\label{omega}\omega=\dfrac{g'h}{gh'}.\end{equation} The logharmonic mappings has been largely studied in the past, for instance we can found the classical results of this topic in \cite{AB88,MPW13}.
\par
In this paper we shall consider the logharmonic mappings defined by equation (6) in the unit disc and such that $f=h\overline{g}:\D\to\C$ where $h$ is an analytic locally univalent mapping and $g$ not vanish in $\D$, which occur when $f$ is locally univalent. These two previous examples satisfies this type of conditions.

\section{Pre-Schwarzian derivative for non vanishing sense-preserving logharmonic mappings}

\subsection{Previous definition of pre-Schwarzian derivative} For locally univalent holomorphic mappings $f:\C\to\C$ the definition of the pre-Schwarzian $P_f$ is given by $$P_f(z)=\dfrac{f''}{f'}(z)=\dfrac{\partial}{\partial z}\log(J_f)=\dfrac{\partial}{\partial z}\log|f'(z)|^2.$$ Thus, $P_f=0$ if and only if $f$ is an affine mapping and the classical univalence Becker's criterion is proved in \cite{B72} which assert that $|P_f(z)|(1-|z|^2)\leq 1$ then $f$ is univalent in the unit disc. Moreover, when $f$ is a locally sense-preserving harmonic mapping, in which case $f=h+\overline{g}$ where the second complex dilatation is given by $\omega=g'/h'$ and the Jacobian is $J_f=|h'|
^2-|g'|^2=|h'|^2(1-|\omega|^2)$, the pre-Schwarzian derivative is defined as $$P_f=\dfrac{h''}{h'}-\dfrac{\omega'\overline{\omega}}{1-|\omega|^2}=\dfrac{\partial}{\partial z}\log(J_f)=\dfrac{\partial}{\partial z}\log(|h'|^2(1-|\omega|^2)).$$ For instance, when $P_f=0$ then $\omega$ is a constant and $h(z)=az+b$, moreover, in \cite{RH-MJ} the authors proved an extension of univalence Becker's criterion in the context of harmonic mappings. In the same way, if $P_f$ is an harmonic mapping, the only option is that the mappings will be holomorphic, in which case, $\omega$ is a constant and $P_h=P_f$.

\subsection{Definition of the $P_f$ for logharmonic mappings}
As we mentioned before, we shall consider a locally univalent logharmonic mapping $f$ defined in the unit disc and that can be written as $f=h\overline g$ where $h$ and $g$ are analytic functions in $\D$. Using equation (\ref{def-log}) and (\ref{jacobiano}), $f$ is locally univalent then $h$ locally univalent and $g$ does not vanishes in the unit disc and $\omega(\D)\subset\D$.
\begin{defn} Let $f=h\overline g$ be a locally univalent logharmonic mapping in a simple connected domain. We define the pre-Schwarzian derivative of $f$ by:
\begin{equation}\label{pf} P_f=\dfrac{h''}{h'}+\omega \dfrac{h'}{h}-\dfrac{\omega'\overline\omega}{1-|\omega|^2}\,.
\end{equation}
\end{defn}
Note that as in the analytic and harmonic case, the Preschwarzian derivative of $f$ satisfies that $P_f=\partial\log J_f/\partial z$. Moreover, when $\omega=0$ this definition coincides with the analytical one.

\begin{ej} Consider $f(z)=z^a\overline{z^b}$ for some complex number $a$ and $b$, and $z$ in any simple connected domain of $\C$ where the origin is not included. Now, a straightforward calculation gives that the dilatation is a constant, this is, $\omega=b/a$. For our purpose we consider $b/a\in\D$, in which case this function is a sense-preserving logharmonic mapping. Hence, \begin{equation*}\label{ej} P_f(z)=\dfrac{a+b-1}{z}.\end{equation*} Thus $P_f=0$ if and only if $a+b=1$ which holds when Re$\{a\}\geq 1/2$.
\end{ej}

\noindent\textbf{Chain Rule:} Let $f=h\overline{g}$ be a sense-preserving logharmonic mapping in a simply connected domain $\Omega\subset\C$ where $h$ and $g$ are non vanishing. It is well-known that if $\varphi$ is a locally univalent analytic function for which the composition $f\circ\varphi$ is defined, then the function $f\circ \varphi$ is again a sense-preserving logharmonic mapping. A straightforward calculation shows that \begin{equation*} P_{f\circ \varphi}=P_f(\varphi)\varphi'+P\varphi.\end{equation*}

\begin{prop} Let $f=h\overline g$ be a sense-preserving locally univalent logharmonic mapping defined in a simple connected domain with dilatation $\omega$ and such that $P_f=0$, then $$h(z)=(az+b)^k,\quad \mbox{and}\quad \omega(z)=\dfrac{1-k}{k},$$ where $a$, $b$, and $k$ are complex numbers with $\mbox{Re}\{k\}>1/2$.
\end{prop}
\begin{proof} As $P_f=0$ we have that $$P_f=\dfrac{h''}{h'}+\omega \dfrac{h'}{h}-\dfrac{\omega'\overline\omega}{1-|\omega|^2}=0\,,$$ therefore, differentiating with respect to $\overline z$ we conclude that $$\dfrac{|\omega'|^2}{(1-|\omega|^2)^2}=0,$$ hence $\omega$ is a complex number in the unit disc, named $(1-k)/k$. As $\mbox{Re}\{k\}>1/2$ it follows that $(1-k)/k$ belongs to the unit disc. Thus, \begin{equation}\label{prop1}\dfrac{h''}{h'}+\left(\dfrac{1-k}{k}\right)\dfrac{h'}{h}=0\end{equation} or equivalently, $$\dfrac{\partial}{\partial z}\log(h'h^{\frac{1-k}{k}})=0.$$ We can conclude that $h'h^{(1-k)/k}$ is a constant, which implies that  $h^{1/k}$ is lineal, i.e. $$h(z)=(az+b)^k.$$ \end{proof}
\begin{obs}
On the other hand, $g'/g=\omega h'/h$ then $g=ch^{1/k-1}$ for some complex number $c$. Therefore $$g(z)=c(az+b)^{1-k}.$$
\end{obs}

\begin{prop} Let $f=h\overline{g}$ be a sense-preserving locally univalent logharmonic mapping defined in simple connected domain and $\varphi$ an holomorphic mapping. Then $Pf=P\varphi$ if and only if $$h(z)=\left(c\varphi+d\right)^k,\quad  \mbox{and} \quad \omega=\dfrac{1-k}{k},$$ for some constants $c$, $d$, and $k$, where $\mbox{Re}\{k\}>1/2$.
\end{prop}

\begin{proof} Since $P_f=P\varphi$, differentiating the equation (\ref{pf}) with respect to $\overline z$, as in the previous proposition, we have that $\omega$ is a complex number, named $(1-k)/k$, thus using equation (\ref{prop1}), we have that $$\dfrac{\partial}{\partial z}\log(h'h^{\frac{1-k}{k}})=\dfrac{\varphi''}{\varphi'}=\dfrac{\partial}{\partial z}\log(\varphi'),$$ therefore $h'h^{(1-k)/k}=a\varphi'$ for some $a\in\C$. Integrating one more time, we obtain that $$h^{1/k}=\dfrac{a}{k}\varphi+\frac{b}{k},$$ for some $b\in\C$, from where the results it follows.
\end{proof}

\begin{obs}
On the other hand, $g'/g=\omega h'/h$ then $g=ah^{1/k-1}$ for some complex number $a$. Therefore $$g(z)=(m\varphi+n)^{1-k},$$ for some complex numbers $m$ and $n$.
\end{obs}

\subsection{Best affine approximation}

Let $f=h\overline g$ be a sense-preserving logharmonic mapping defined in a simple connected domain that contains the origin. We construct a logharmonic mapping $T$ defined as follows:  $T(z)=L(z)\overline{L(z)^{1-k/k}}$ where $L(z)=(az+b)^{k}$ and such that
\begin{equation}\label{normalization}
  T(0)=f(0),\quad \dfrac{\partial T}{\partial z}(0)=\dfrac{\partial f}{\partial z}(0),\quad \mbox{and} \quad\dfrac{\partial T}{\partial \overline z}(0)=\dfrac{\partial f}{\partial \overline z}(0).
\end{equation} Thus, $\log T=\log L+c\overline{\log L}$ where $c=\overline{(1-k)/k}$ is a sense-preserving harmonic mapping. It is also an affine harmonic mapping. Since $T$ is an univalent logharmonic mapping, and the best affine logharmonic mapping approximation of $f$, we define $F=T^{-1}\circ f$, and the normalization given in (\ref{normalization}) we have that \begin{equation}\label{cond-inicial}F(0)=0,\quad \dfrac{\partial F}{\partial z}(0)=1,\quad \dfrac{\partial F}{\partial \overline z}(0)=0.\end{equation} By the definition we can compute that
\begin{eqnarray}
F_z & = & \left(\dfrac{\overline{T_z}}{J_T}\right)(F)f_z-\dfrac{T_{\overline{z}}}{J_T}(F)\left(\overline{f}\right)_z\,, \label{e1}\\
F_{\overline{z}} & = & \left(\dfrac{\overline{T_z}}{J_T}\right)(F)f_{\overline{z}}-\dfrac{T_{\overline{z}}}{J_T}(F)\left(\overline{f}\right)_{\overline{z}}\,.\label{e2}
\end{eqnarray} Differentiating $(\ref{e1})$ and $(\ref{e2})$  with respect to $z$ and $\overline z$ and evaluating at the origin, from equations in (\ref{cond-inicial}) it follows that: \begin{equation}
F_{zz}(0)=\dfrac{(f_{zz}(0)-T_{zz}(0))\overline{f_z(0)}-\overline{(f_{\overline z\, \overline z}(0)-T_{\overline z\,\overline z}(0))}f_{\overline{z}}(0)}{J_f(0)}.\label{e3}
\end{equation}

\begin{equation}
F_{z\overline z}(0)=\dfrac{(f_{z\overline z}(0)-T_{z\overline z}(0))\overline{f_z(0)}-\overline{(f_{z\overline z}(0)-T_{z\overline z}(0))}f_{\overline{z}}(0)}{J_f(0)}.\label{e4}
\end{equation}

\begin{equation}
F_{\overline z\,\overline z}(0)=\dfrac{(f_{\overline z\,\overline z}(0)-T_{\overline z\,\overline z}(0))\overline{f_z(0)}-\overline{(f_{zz}(0)-T_{zz}(0))}f_{\overline{z}}(0)}{J_f(0)}.\label{e5}
\end{equation} By the definition of $T$ we have that \begin{equation}T_{zz}(0)=-\omega(0)\frac{f_z(0)^2}{f(0)},\quad T_{z\overline z}(0)=\frac{f_z(0)f_{\overline z}(0)}{f(0)},\quad T_{\overline z\,\overline z}(0)=-\frac{\overline{f_z(0)}f_{\overline z}(0)}{\overline{f(0)}}.\label{e6}\end{equation} Using equations (\ref{e3}), (\ref{e4}),(\ref{e5}), and (\ref{e6}), a straightforward calculation gives that the Taylor expansion of $F$ in terms of $z$ and $\overline z$ is $$F(z)=z+P_f(0)z^2+ \dfrac{e^{i\theta}\overline{\omega'(0)}}{1-|\omega(0)|
^2}\overline z^2+\cdots,$$ where $e^{i\theta}=h(0)\overline{h'(0)}/h'(0)\overline{h(0)}$ and $\partial P_f/\partial \overline z=|\omega'|^2/(1-|\omega|^2)^2$.

\begin{thm}Let $f=h\overline g$ be a nonvanishing sense-preserving logharmonic mapping with dilatation $\omega$. For any $z_0\in D$ we have that$$P_f(z_0)=P(h^ag^b)(z_0),$$ where $a=\dfrac{1+\omega(z_0)}{1-|\omega(z_0)|^2}$ and  $b=-\overline{\omega(z_0)}\left[\dfrac{1+\omega(z_0)}{1-|\omega(z_0)|^2}\right]$.
\end{thm}

\begin{proof} Since $h$ and $g$ are nonvanishing mappings, then $\varphi=h^ag^b$ is a nonvanishing analytic function such that
\begin{eqnarray*}
\varphi' & = & ah^{a-1}h'g^b+bh^ag^{b-1}g'\\
& = & ah^{a-1}h'g^b\left[1+\dfrac{b}{a}\dfrac{hg'}{h'g}\right]\\
& = &  ah^{a-1}h'g^b\left[1+\dfrac{b}{a}\omega\right],\\
\end{eqnarray*}
then
\begin{eqnarray*}
\dfrac{\varphi''}{\varphi'} (z_0)& = & (a-1)\dfrac{h'}{h}(z_0)+(b-1)\dfrac{g'}{g}(z_0)+ \dfrac{h''}{h'}(z_0)+\dfrac{g'}{g}(z_0)+\dfrac{\dfrac{b}{a}\omega'(z_0)}{1+\dfrac{b}{a}\omega(z_0)}\\\\
& = & (a-1)\dfrac{h'}{h}(z_0)+(b-1)\dfrac{g'}{g}(z_0)+ \dfrac{h''}{h'}(z_0)+\dfrac{g'}{g}(z_0)-\dfrac{\overline{\omega(z_0)}\omega'(z_0)}{1-|\omega(z_0)|^2}\\\\
& = & (a-1)\dfrac{h'}{h}(z_0)+(b-1)\dfrac{g'}{g}(z_0)+P_f(z_0).\\
\end{eqnarray*} As $a=\dfrac{1+\omega(z_0)}{1-|\omega(z_0)|^2},$ so $a(1-|\omega(z_0)|^2)=1+\omega(z_0)$ and from this $ (a-1)\dfrac{h'}{h}+(b-1)\dfrac{g'}{g}=0,$ therefore $$P(\varphi)(z_0)=P_f(z_0).$$\\
\end{proof}

\begin{prop} Let $f=h\overline{g}$ be a sense-preserving logharmonic mapping defined in a simple connected domain, and $L(z)=z^a\overline{z^b}$, $a\neq0$. Then $P_f=P_{L\circ f}$ when $a+b=1$.
\end{prop}

\begin{proof} Let $F=L\circ f$, then $
F = (h\overline{g})^a\overline{(h\overline{g})^b}=h^ag^{\overline{b}}\,\overline{g^{\overline{a}}h^b}= H\overline{G}$, where $H$ and $G$ are nonvanishing analytic mappings, and its dilatation is given by $$ W  =  \dfrac{G'H}{GH'}=\dfrac{[\overline{a}g^{\overline{a}-1}g'h^b+bh^{b-1}h'g^{\overline{a}}]h^ag^{\overline{b}}}{[ah^{a-1}h'g^{\overline{b}}+h^a\overline{b}g^{\overline{b}-1}g']g^{\overline{a}}h^b} =\dfrac{\overline{a}}{a}\left[\dfrac{g'h+\varepsilon h'g}{h'g+\overline{\varepsilon} hg'}\right]=\dfrac{\overline{a}}{a}\left[\dfrac{\omega+\varepsilon}{1+\overline{\varepsilon}\omega}\right], $$ where $\varepsilon=b/\overline a$. Note that $$H'  =  ah^{a-1}h'g^{\overline{b}}+\overline{b} h^ag^{\overline{b}-1}g'=ah^{a-1}g^{\overline{b}-1}h'g\left[1+\dfrac{\overline{b}}{a}\cdot\dfrac{hg'}{h'g}\right]= ah^{a-1}g^{\overline{b}-1}h'g\left[1+\varepsilon\omega\right].$$ Therefore $$\dfrac{H''}{H'}=(a-1)\dfrac{h'}{h}+(\overline{b}-1)\dfrac{g'}{g}+\dfrac{h''}{h'}+\dfrac{g'}{g}+\dfrac{\overline{\varepsilon}\omega'}{1+\overline{\varepsilon}\omega}.$$ On the other hand $G=g^{\overline{a}}h^b$ since $\dfrac{G'}{G}=\overline{a}\dfrac{g'}{g}+b\dfrac{h'}{h},$ then
\begin{eqnarray*}
P_{F} &=& \dfrac{H''}{H'}+\dfrac{G'}{G}-\dfrac{W'\overline{W}}{1-|W|^2}\\
& = & (a-1)\dfrac{h'}{h}+(\overline{b}-1)\dfrac{g'}{g}+\dfrac{h''}{h'}+\dfrac{g'}{g}+\dfrac{\overline{\varepsilon}\omega'}{1+\overline{\varepsilon}\omega}+\overline{a}\dfrac{g'}{g}+b\dfrac{h'}{h}-\dfrac{W'\overline{W}}{1-|W|^2}\\
& = & (a+b-1)\dfrac{h'}{h}+(\overline{a}+\overline{b}-1)\dfrac{g'}{g}+\dfrac{\overline{\varepsilon}\omega'}{1+\overline{\varepsilon}\omega}-\dfrac{\omega'(\overline{\omega+\varepsilon})}{(1+\overline{\varepsilon}\omega)(1-|\omega|^2)}+\dfrac{h''}{h'}+\dfrac{g'}{g}\\
& = & (a+b-1)\dfrac{h'}{h}+(\overline{a}+\overline{b}-1)\dfrac{g'}{g}-\dfrac{\omega'\overline{\omega}}{1-|\omega|^2}+\dfrac{h''}{h'}+\dfrac{g'}{g}\\
& = & (a+b-1)\dfrac{h'}{h}+(\overline{a+b-1})\dfrac{g'}{g}+P_f.
\end{eqnarray*}
This concludes that $P_F=P_f$ when $a+b=1$.
\end{proof}

\begin{thm} Let $f=h\overline g$ and $F=H\overline{G}$ be non vanishing logharmonic sense-preserving mappings with dilatations $\omega_f$ and $\omega_F$ respectively. If $P_f=P_F$ then $\omega_F=\lambda\omega_f$ for some complex number $\lambda$, with $|\lambda|=1$.
\end{thm}

\begin{proof} Since $P_f=P_F$,  differentiating with respect to $\overline{z}$ we can obtain that \begin{eqnarray}
\dfrac{|\omega'_f|^2}{(1-|\omega_f|^2)^2}=\dfrac{|\omega'_F|^2}{(1-|\omega_F|^2)^2}.\label{eq2}
\end{eqnarray} By the last proposition, we can suppose that $\omega_f(0)=\omega_F(0)=0$. Thus, evaluating this equation at the origin, it follows that  $|\omega_f'(0)|=|\omega_F'(0)|$ and assuming that this quantity is different of zero, we have that $\omega_f'(0)=\lambda\omega_F'(0)$ for some $\lambda$ in the unit circle. Now, differentiating the equation (\ref{eq2}) with respect to $z$ and evaluating at the origin, we have that $$\omega_f''(0)\overline{\omega_f'(0)}=\omega_F''(0)\overline{\omega_F'(0)}=\omega_F''(0)\overline{\lambda\omega_f'(0)},$$ therefore $\omega_f''(0)=\overline \lambda\omega_F''(0)$. Using induction we can show that $$\frac{\partial^n \omega_f}{\partial z^n}(0)=\overline{\lambda}\frac{\partial^n \omega_F}{\partial z^n}(0),$$ hence $\omega_f=\overline\lambda\omega_F$.
\end{proof}

\begin{obs} Since the definition of the pre-Schwarzian, we have that if $P_f=P_F$ therefore $$\dfrac{\partial}{\partial z}\log(h'g)=\dfrac{\partial}{\partial z}\log(H'G),$$ thus, there exists a complex number $c$ such that $h'g=cH'G$ and, hence $\omega_f=\lambda\omega_F$ for some $\lambda$ in the unit circle, we have that $hg'=\overline \lambda c HG'$.
\end{obs}

\section{Univalence criteria}

In the context of analytic and harmonic functions defined in the unit disc or in any simple connected domain, the concept of \textit{stability} was introduced by Hern\'andez and Mart\'in in \cite{rhmje} but this concept appears in a natural way in the celebrated paper by Clunie and Sheil-Small \cite{CSS}. In fact, the authors defined that $h+g$ is a stable analytic univalent function if $h+\lambda g$ is univalent for all $|\lambda|=1$, and they proved that this is equivalent with $h+\lambda \overline{g}$ is univalent for all $|\lambda|=1$. Thus, $h+g$ is a stable analytic univalent mapping if and only if $h+\overline{g}$ is a stable harmonic univalent mapping. The family of univalent mappings can be replaced by the Convex or Starlike mappings in the definition of stability. The reader can find all those details in \cite{rhmje}.

\begin{thm} Let $f=h\overline g:\D\to\C$ be a nonvanishing logharmonic sense-preserving mapping with dilatation $\omega$. If \begin{equation}\label{univ}\left|zP_f(z)\right|+(1+|\omega|)\left|z\dfrac{h'}{h}(z)\right|+\dfrac{|z\omega'(z)|}{1-|\omega(z)|^2}\leq\dfrac{1}{1-|z|^2},\end{equation} then $f$ is univalent. The constant 1 is sharp.
\end{thm}

\begin{proof} Let $z_1$ and $z_2$ in $\D$ such that $f(z_1)=f(z_2)$. Since $f\neq 0$ in the unit disc, then $F=\log(f)=\log h+\log \overline g=\log h+\overline{\log g}$ is a well-defined harmonic mapping in $\D$, such that $F(z_1)=F(z_2)$. However $F=H+\overline G$ satisfy that $\omega_F=G'/H'=\omega$ and $$P_F=\dfrac{H''}{H'}-\dfrac{\omega_F'\overline \omega_F}{1-|\omega_F|^2}=\dfrac{h''}{h'}-\dfrac{h'}{h}-\dfrac{\omega'\overline \omega}{1-|\omega|^2}=P_f-(1-\omega)\dfrac{h'}{h}.$$ Therefore, $$|zP_F|+\frac{|z\omega'|}{(1-|\omega|
^2)}\leq|zP_f|+(1+|\omega|)\left|z\dfrac{h'}{h}\right|+\dfrac{|z\omega'|}{1-|\omega|^2}.$$ Thus, by equation (\ref{univ}) and Becker's criterion of harmonic mapping (\cite[thm. 8]{RH-MJ}) $F$ is univalent, then $z_1=z_2$ which implies that $f$ is univalent. Since the constant 1 is sharp for harmonic functions, it follows that this constant is sharp for the logharmonic mappings, too.
\end{proof}

Moreover, the function $f_\lambda=h\overline{g^\lambda}$ satisfies equation (\ref{univ}) with $|\lambda|=1$ since the analytic part is the same and $\omega_{f_\lambda}=\lambda\omega_f$, therefore $f_\lambda$ is univalent in the unit disc for all $\lambda$ in the unit circle, which means that $f$ is a sense-preserving stable logharmonic univalent mapping.

\begin{thm}Let $f=h\overline g:\D\to\C$ be a nonvanishing logharmonic sense-preserving mapping with dilatation $\omega$ and $c\in\C$ such that $|c|\leq 1$ with $c\neq -1$. If  \begin{equation}\label{univ2}\left|(1-|z|^2)zP_f(z)+c|z|^2\right|+(1+|\omega|)\left|(1-|z|^2)z\dfrac{h'}{h}(z)\right|+\dfrac{|z\omega'(z)|(1-|z|^2)}{1-|\omega(z)|^2}\leq1,\end{equation} then $f$ is univalent. The constant 1 is sharp.
\end{thm}

\begin{proof} Since $f$ is a nonvanishing logharmonic mapping, it follows that $F=\log f$ is an harmonic mapping in the unit disc, such that $P_F=P_f-(1+\omega)h'/h$. Thus $F$ satisfies the hypothesis of Proposition 1 in \cite{VBRHOV1}, then $F$ is univalent harmonic mapping, therefore $f$ is univalent logharmonic function.
\end{proof}

\begin{obs} If $|P_f(z)|(1-|z|^2)+|\omega'|(1-|z|^2)/(1-|\omega|^2)\leq 1$ then the analytic function $\varphi$, defined such that $\varphi'=h'g$ satisfies that $|P_\varphi(z)|(1-|z|^2)\leq 1$, then $h'g$ is a univalent mapping in the unit disc. Using Proposition 3 one can conclude that $\varphi_\varepsilon$ is univalent too  when $\varphi_\varepsilon'=h'g+\overline \varepsilon g'h$. Taking $a$ in the line Re$\{a\}=1/2$ and $b=1-a$ we have $|\varepsilon|=1$, thus for $\varepsilon=1$ we obtain that $hg$ is an analytic univalent mapping, in fact, this argument shows that $\varphi_\varepsilon$ is a stable analytic univalent mapping, as a consequence $$\int_0^zh'(\zeta)g(\zeta)d\zeta+\overline{\int_0^zh(\zeta)g'(\zeta)d\zeta}\,,$$ is a univalent harmonic mapping in the unit disc where its dilatation is the dilatation of $f$.
\end{obs}


\begin{thebibliography}{11}

\bibitem{AB88} Z. Abdulhadi and D. Bshouty, Univalent functions in $H.\overline H(\D)$, Trans. Am. Math. Soc. 305(2) (1988), pp. 841--849.\medskip

\bibitem{AR} Z. Abdulhadi and R.M. Ali, Univalent logharmonic mappings in the plane, Abstr. Appl. Anal. 2012 (2012), Art. ID 721943, 32 pp.\medskip

\bibitem{Ah}  L. Ahlfors, Sufficient condition for quasiconformal extension, Ann. Math. Studies. 79 (1974) 23--29.\medskip

\bibitem{B72} Becker~J. L\"{o}wnersche
    Differentialgleichung und quasi-konform forttsetzbare schlichte Funktionen. J. Reine. Angew.
    Math. 1972; 225:23--43.\medskip

\bibitem{BP} J. Becker and Ch. Pommerenke, Schlichtheitskriterien und Jordangebiete, J. Reine Angew. Math.,
354, 74--94 (1984).\medskip

\bibitem{VBRHOV1}Bravo V., Hern\'andez R., and Venegas O. On the univalence of certain integral transform for harmonic mappings, J. Math. Anal. Appl., 2017; 455(1), 381--388.\medskip

\bibitem{CHDO}
M. Chuaqui, P. Duren, and B. Osgood, The Schwarzian derivative for
harmonic mappins, J. Anal. Math., 91, 329--351 (2003).\medskip

\bibitem{CSS}
J. Clunie and T. Sheil-Small,
Harmonic univalent functions, Ann. Acad. Sci. Fenn. Ser. A.I., 9, 3--25(1984).\medskip

\bibitem{Dur-Univ}
P. Duren, Univalent Functions, Springer-Verlag, New York (1983).\medskip

\bibitem{Dur-Harm}
P. Duren, Harmonic Mappings in the Plane, Cambridge University
Press, Cambridge (2004).\medskip

\bibitem{EFHV}
I. Efraimidis, A. Ferrada-Salas, R. Hern\'andez, and R. Vargas, Schwarzian derivatives for Pluriharmonic mappings, https://arxiv.org/pdf/1912.12619.pdf\medskip

\bibitem{GK}
I. Graham and G. Kohr, Geometric function theory in one and
higher dimensions, Marcel Dekker Inc., New York, Basel (2003).\medskip

\bibitem{RH-MJ} Hern\'andez~R., Mart\'in~M.J. Pre-Schwarzian and Schwarzian Derivatives of Harmonic Mappings. J. Geom. Anal., 25(1); 64--91(2015).\medskip

\bibitem{rhmje} Hern\'andez~R., Mart\'in~M.J. Stable geometric properties of analytic and harmonic functions. Math. Proc. Cambridge Philos. Soc. 2013; 155(2):343--359.\medskip

%

\bibitem{MPW13} Zh. Mao, S. Ponnusamy, and X. Wang, Schwarzian derivative and Landau's theorem for logharmonic mappings, Complex Var. Elliptic Equ. 58(8) (2013), 1093--1107.\medskip

\bibitem{LP} Z. Liu and S. Ponnusamy, On the univalet log-harmonic mappings, 	arXiv:1905.10551 [math.CV].\medskip

\bibitem{P}
Ch. Pommerenke, Univalent Functions, Vandenhoeck $\&$
Ruprecht, G\"{o}ttingen (1975).\medskip

\bibitem{T}
H. Tamanoi, Higher Schwarzian operators and combinatorics of the
Schwarzian derivative, Math. Ann., 305, 127--151 (1996).\medskip



\end{thebibliography}
\end{document}